\documentclass{mystyle} 
\usepackage{latexsym} 
\usepackage{alltt}
\usepackage[all]{xy} 
\CompileMatrices 
\xyoption{frame}
\usepackage{amssymb}
\usepackage{amscd} 
\usepackage{amsmath}
\usepackage{theorem} 
\theoremstyle{break} \newtheorem{theorem}{Theorem}[section]
\theoremstyle{break} \newtheorem{corollary}[theorem]{Corollary}
\theoremstyle{break} \newtheorem{lemma}[theorem]{Lemma}
\theoremstyle{break} \newtheorem{proposition}[theorem]{Proposition}
\theoremstyle{plain} {\theorembodyfont{\rmfamily}\newtheorem{definition}[theorem]{Definition}}
\theoremstyle{plain} {\theorembodyfont{\rmfamily} }
\theoremstyle{plain} {\theorembodyfont{\rmfamily} }

\newcommand{\Nat}{{\mathbb N}}

\newcommand{\num}[1]{\underline{#1}}
\newcommand{\rat}[2]{\underline{#1}\cdot\underline{#2}^{-1}}

\begin{document}
\title[The initial meadows]{The initial meadows}
\author[Inge Bethke and Piet Rodenburg]{Inge Bethke
and Piet Rodenburg\\
University of Amsterdam, 
Faculty of Science, \\
Section Theoretical
Software Engineering (former Programming Research Group)} 
\maketitle
\noindent {\bf Abstract}: A \emph{meadow} is a commutative ring with an inverse operator satisfying $0^{-1}=0$. We determine
the initial algebra of the meadows of characteristic 0 and show that its word problem is decidable.\\

\mbox{}\\
\mbox{}

\section{Introduction}
A \emph{field} is a fundamental algebraic structure with total operations of addition, subtraction
and multiplication. Division, as the inverse of multiplication, is subjected to the restriction that
every element has a multiplicative inverse|except 0. In a field, the rules 
hold
which are familiar from the arithmetic of ordinary numbers.
That is, fields  can be specified by the axioms for 
commutative rings with identity element
($\mathit{CR}$, see Table \ref{CRaxioms}), and the negative conditional formula 
\[
x\neq 0 \rightarrow x\cdot x^{-1}=1.
\]
The prototypical example is the field of rational numbers.
\begin{table}\label{CRaxioms}
\[
\begin{array}{lrcl}
\hline
&(x+y)+z &=&x+(y+z)\\
&x+y&=&y+x\\
&x +0 &=& x\\
&x+(-x)&=&0\\
&(x \cdot y)\cdot z&=& x\cdot (y \cdot z)\\
&x\cdot y &=&y \cdot x\\
&x \cdot 1 &=& x\\
&x\cdot (y + z) &=& x\cdot y + x \cdot z\\
\hline
\end{array}
\]
\caption{Specification $CR$ of commutative rings with multiplicative identity}
\end{table}

In Bergstra and Tucker (2007) the name \emph{meadow} was proposed for commutative 
rings with a multiplicative identity element and a total operation  
$\ ^{-1}$|inversion|governed by \emph{reflection} and the \emph{restricted inverse law}.
\begin{table}
\[
\begin{array}{lrcl}
\hline
(\mathit{Ref}) &(x^{-1})^{-1} &=& x\\
(\mathit{Ril}) & x \cdot (x \cdot x^{-1})& = &x\\
\hline
\end{array}
\]
\caption{Reflection and restricted inverses law\label{Refrilaxioms}}
\end{table}
We write $\mathit{Md}$ for the set of axioms in Table \ref{CRaxioms} augmented by the additional equations in Table \ref{Refrilaxioms}.
In fact, Bergstra and Tucker (2007) requires in addition that $(-x)^{-1}=-x^{-1}$ and $(x\cdot y)^{-1}=
x^{-1}\cdot y^{-1}$. Those equations have been shown derivable from $\mathit{Md}$.

From the axioms in $\mathit{Md}$ the following identities are derivable (cf. Bergstra et al. (2007), 
Bergstra et al. (2008)).
\[
\begin{array}{rcl}
0^{-1} &=& 0\\
(-x)^{-1}&=&-(x^{-1})\\
(x\cdot y)^{-1}&=&x^{-1}\cdot y^{-1}\\
x\cdot 0 &=&0\\
x\cdot -y&=&-(x\cdot y)\\
-(-x)&=&x\\
\end{array}
\]
One can also e.g.\ show that a meadow has no nonzero nilpotent elements: Suppose $x\cdot x = 0$. Then
\[
x = x \cdot (x \cdot x^{-1})= (x \cdot x) \cdot x^{-1} = 0 \cdot x^{-1} =  x^{-1}\cdot 0 = 0.
\]
Fields are meadows if we complete the inversion operation by 
$0^{-1} = 0$. The result is called a \emph{zero-totalized field}. 

When abstract data types are specified algebraically, the \emph{initial algebra} is often taken as 
the meaning of the specification. The initial algebra always exists, is unique up to isomorphism, and 
can be constructed from the closed term algebra by dividing out over provable equality. 
Some references to universal algebra and initial algebra semantics are e.g.\
Goguen et al. (1977), Gr\"atzer (1977), McKenzie et al. (1987) and Wechler (1992).

The initial meadows of finite characteristic $k>0$ have been described already: in Bergstra et al. (2007) it is proved
that $k$ must be squarefree and that the initial meadow  of charateristic $k$ has 
$p_1\cdots p_n$ elements, where $p_1, \ldots , p_n$ are the distinct prime factors of $k$. It then follows from
Corollary 2.9 in
Bethke and Rodenburg (2007) that the initial meadow is isomorphic with
$\mathbb{G}_{p_1}\times \cdots \times \mathbb{G}_{p_n}$ where $\mathbb{G}_{p_i}$
is the prime field of order $p_i$.

In this paper we represent the initial meadow of characteristic $0$ as the minimal subalgebra of the direct product of all finite prime fields and 
show that its word problem is decidable.
Theorem \ref{initialmeadow} stems from a suggestion made by Yoram Hirshfeld, Tel Aviv University, in a private communication.
The decidability result is a rigorous elaboration of a remark made in Bergstra and Tucker (2007)|in the proof of 
Corollary 5.11|and can be read between the lines in their Section 5.

\section{The initial meadow of characteristic 0}
In this section we shall show that the initial meadow is a proper subdirect product of all prime fields.
\begin{definition}
\begin{enumerate}[(i)]
\item A \emph{subdirect embedding} of a meadow $M$ in a family
$(M_j)_{j\in J}$ of meadows is a family $(\phi_j:M \twoheadrightarrow M_j)_{j\in J}$ of surjective homomorphisms
such that for any distinct  $x,y\in M$ there exists $j\in J$ such that $\phi_j(x)\neq \phi_j(y)$. 
\item We say $M$ 
is a \emph{subdirect product} of $(M_j)_{j\in J}$ if $M\subseteq \Pi_{j\in J}M_j$ and the restricted 
projections $M\rightarrow M_j$ form a subdirect embedding of $M$.
\item A meadow $M$ is called \emph{subdirectly irreducible} when every subdirect embedding of $M$ 
contains an isomorphism.
\end{enumerate}
\end{definition}
Loosely speaking, this means that a meadow is subdirectly irreducible when it cannot be represented as
a subdirect product of ``smaller" meadows, i.e.\ proper epimorphic images. An instance of Birkhoff's Subdirect
Decomposition Theorem (see Birkhoff (1944) and Birkhoff (1991)) states
\begin{enumerate}[(i)]
\item \emph{Every meadow is isomorphic with a subdirect product of subdirectly irreducible meadows.}
\end{enumerate}
If we forget the multiplicative identity element and the inversion operation in a given 
meadow, what remains is a commutative ring satisfying
\[
\begin{array}{rrcl}
\hline
\exists x \forall y & x\cdot y& = &y ,\\
\forall x \exists y &\ x\cdot x\cdot y& =& x ,\\
\hline
\end{array}
\]
a commutative \emph{regular ring} in the sense of Von Neumann (see Goodearl (1979)). It is not hard to 
see that the $x$ in the first formula is unique, and it is shown in Bergstra et al. (2007) and Bergstra et al. (2008) that 
for any $x$, there is a unique $y$ such that both $x\cdot x\cdot y = x$ and 
$y\cdot y\cdot x = y$. So a commutative regular ring determines a unique meadow, and vice versa. 
Since $x^{-1} = x^{-1}\cdot x^{-1} \cdot x$, the 
ideals of a meadow are closed under inversion, so that in meadows, as in 
rings, ideals correspond completely to congruence relations. 
As a consequence, the lattice of congruence relations of the ring reduct of a meadow
coincides with the lattice of congruence relations of the meadow. 
We may therefore restate Lemma 2 of Birkhoff (1944)
as follows:
\begin{enumerate}[(i)]\addtocounter{enumi}{1}
\item \emph{A subdirectly irreducible meadow is a zero-totalized field.}
\end{enumerate}
Combining (1) and (2), we have that the initial meadow lies subdirectly embedded in a product 
of subdirectly irreducible zero-totalized fields. We may assume that every factor occurs only 
once|we still have a representation if we remove doubles. All factors are 
minimal, since they are homomorphic images of a minimal algebra. The 
minimal zero-totalized fields are the prime fields $\mathbb{G}_p$, $p$ a prime number, and 
$\mathbb{Q}$, the rational numbers.
\begin{lemma}
Let $A$ be the minimal subalgebra of the direct product $\mathbb{G}:= \prod_{p\text{ prime}} \mathbb{G}_p$. 
Let $Z_p$ be the element of $\mathbb{G}$ that is $0$ in all coordinates except $p$, where it is $1$. Then
\begin{enumerate}[(i)]
\item $Z_p \in A$,
\item the direct sum $\sum_{p\text{ prime}}\mathbb{G}_p$ lies embedded as an ideal in $A$, and
\item if we identify $\sum_{p\text{ prime}}\mathbb{G}_p$ with its image in $A$, 
$A/\sum_{p\text{ prime}}\mathbb{G}_p\cong \mathbb{Q}$.
\end{enumerate}
\end{lemma}
\begin{proof} 
(1) $Z_p$ is the denotation of the ground term $1- \underline{p}\cdot \underline{p}^{-1}$, where 
$\underline{p}$
stands for the ground term $1 + \cdots + 1$, with $p$ occurrences of $1$.\\
(2) Modulo isomorphism,  $\sum_{p\text{ prime}}\mathbb{G}_p$ is the ideal of $A$ generated by the 
$Z_p$'s. 
If we multiply an element of this ideal with any element of $\mathbb{G}$, the result is almost everywhere zero, and therefore belongs 
to the ideal.\\
(3) $A/\sum_{p\text{ prime}}\mathbb{G}_p$ is a minimal meadow of characteristic 0 that satisfies the 
equations $\underline{n}\cdot \underline{n}^{-1}=1$, for all positive integers $n$. So by Theorem 3.1 of 
Bergstra and Tucker (2007), 
$A/\sum_{p\text{ prime}}\mathbb{G}_p$ is a homomorphic image of $\mathbb{Q}$; since $\mathbb{Q}$ has no 
proper ideals, the homomorphism must be injective.	
\end{proof}
\begin{theorem}\label{initialmeadow}
The minimal subalgebra of $\prod_{p\text{ prime}} \mathbb{G}_p$ is an initial object in the category of meadows.
\end{theorem}
\begin{proof} 
From the observations above, it appears that the initial meadow is the 
minimal subalgebra of the direct product of a set $\mathcal{G}$ of zero-totalized minimal 
fields. It is easily seen that every prime field $\mathbb{G}_p$ must be in $\mathcal{G}$, otherwise 
there is no nontrivial homomorphism from a subalgebra of $\prod \mathcal{G}$ into $\mathbb{G}_p$. So if 
$\mathbb{Q}\not\in \mathcal{G}$, the initial meadow is the algebra $A$ of the previous lemma.
On the other hand, if $\mathbb{Q}\in \mathcal{G}$, by (3) of the lemma we have a surjective homomorphism 
$h: A \rightarrow \mathbb{Q}$. Then $(1, h): A \rightarrow A\times \mathbb{Q}$ shows that 
$A$ must be isomorphic to the minimal 
subalgebra of $\prod \mathcal{G}$.
\end{proof}
The initial meadow is countable, whereas the product of all finite prime fields is uncountable. This cardinality consideration shows that the initial algebra
is properly contained in the product, and is|in contrast to the finite initial meadows|not a product of fields.

\section{Decidability of the closed word problem}
The main result of this section is a rigorous description of normal forms for closed meadow terms.  To be precise, we shall prove that every
closed meadow term $t$ is provably equal to a term of the form
\[
\Sigma_{i=0}^{\psi(t)-1}Z_i\cdot \phi_{i}(t) + G_{\psi(t)}\cdot \phi(t)
\]
 where $\phi_i$ interprets $t$
in the Galois field with order $p_i$ (the $i$-th prime), $\phi$ is its interpretation in the rational numbers, $Z_i$ and $G_i$ select significant models, and 
$\psi(t)$ is an effective upper bound. This is Proposition \ref{normal}. From this it follows immediately , that the closed word problem for meadows is decidable.

We denote by $\Nat$ the set of natural numbers; $Ter_{\mathit{Md}}$ denotes the set of closed meadow terms.

\begin{definition}
\begin{enumerate}[(i)]
\item We define the set of \emph{numerals} $\Nat_{\mathit{Md}}\subseteq Ter_{\mathit{Md}}$ by
\[
\mathbb{N}_{\mathit{Md}}=\{ \underline{n}\mid n\in \Nat\}
\]
where for $n\in \Nat$, $\underline{n}$ is defined inductively as follows:
\begin{enumerate}[(i)]
\item $\underline{0}=0$,
\item $\underline{n+1}=\underline{n}+1$.
\end{enumerate}
\item We define the set of \emph{normal rational terms} $\mathbb{Q}_{\mathit{Md}}\subseteq Ter_{\mathit{Md}}$ by
\[
\mathbb{Q}_{\mathit{Md}}=\{0\}\cup \{\underline{n}\cdot \underline{m}^{-1},-(\underline{n}\cdot \underline{m}^{-1}) \mid n,m\in \Nat\ \&\ n,m>0\ 
\&\  gcd(n,m)=1\}
\]
\end{enumerate}
For $t\in \mathbb{Q}_{\mathit{Md}} $, we denote by $|t|$ the corresponding irreducible fraction in $\mathbb{Q}$.
\end{definition}
Observe that $\underline{\ }$ respects addition, multiplication and subtraction, i.e., $\mathit{Md}\vdash \underline{n}+ \underline{m}=\underline{n+m}$,
$\mathit{Md}\vdash \underline{n}\cdot \underline{m}=\underline{nm}$ and if $m<n$, then $\mathit{Md}\vdash \underline{n}- \underline{m}=\underline{n-m}$.

We now assign to every closed term a normal rational term.
\begin{definition}
We define $\phi: Ter_{\mathit{Md}} \rightarrow \mathbb{Q}_{\mathit{Md}}$ inductively as follows.
\begin{enumerate}[(i)]
\item $\phi(0)=0$, $\phi(1)=\underline{1}\cdot \underline{1}^{-1},$
\item \[
\phi(-t)=
\begin{cases}
0 & \text{if $\phi(t)=0$},\\
-(\num{n}\cdot \num{m}^{-1}) & \text{if $\phi(t)=\num{n}\cdot \num{m}^{-1}$},\\
\num{n}\cdot \num{m}^{-1} & \text{if $\phi(t)=-(\num{n}\cdot \num{m}^{-1})$.}
\end{cases}
\]
\item \[
\phi(t^{-1})=
\begin{cases}
0 & \text{if $\phi(t)=0$},\\
\num{n}\cdot \num{m}^{-1} & \text{if $\phi(t)=\num{m}\cdot \num{n}^{-1}$},\\
-(\rat{n}{m}) & \text{if $\phi(t)=-(\num{m}\cdot \num{n}^{-1})$.}
\end{cases}
\]
\item \[
\phi(t + t')=
\begin{cases}
0 & \text{if $|\phi(t)| + |\phi(t')|=0$},\\
\rat{n}{m} & \text{if $0<|\phi(t)| + |\phi(t')|=\frac{n}{m}$, $n,m>0$ and $gcd(n,m)=1$},\\
-(\rat{n}{m}) & \text{if $0>|\phi(t)| + |\phi(t')|=-\frac{n}{m}$, $n,m>0$ and $gcd(n,m)=1$.}\\
\end{cases}
\]
\item \[
\phi(t \cdot t')=
\begin{cases}
0 & \text{if $|\phi(t)|  |\phi(t')|=0$},\\
\rat{n}{m} & \text{if $0<|\phi(t)| |\phi(t')|=\frac{n}{m}$, $n,m>0$ and $gcd(n,m)=1$},\\
-(\rat{n}{m}) & \text{if $0>|\phi(t)|  |\phi(t')|=-\frac{n}{m}$, $n,m>0$ and $gcd(n,m)=1$.}\\
\end{cases}
\]
\end{enumerate}
\end{definition}
Observe that $\phi$ assigns to provably equal terms syntactically identical normal rational terms. 
\begin{proposition}\label{sounda}
For $s,t \in Ter_{\mathit{Md}}$,
\[
\mathit{Md}\vdash s=t \Rightarrow \phi(s)=\phi(t).
\]
\end{proposition}
\begin{proof} 
Clearly, $|\phi(s)|$ is the interpretation of any closed term $s$ in $\mathbb{Q}$. Thus, if
$s=t$ is derivable, then $|\phi(s)|=|\phi(t)|$, and hence $\phi(s)=\phi(t)$. 
\end{proof}

We can also evaluate closed meadow terms in a finite prime field $\mathbb{G}$. We may think of such 
a field as the ring with the elements $0, 1, 2, \ldots , p-1$, where arithmetic is performed modulo $p$. 
We let $(p_n)_{n\in \Nat}$ be an enumeration of the primes in increasing order, starting with $p_0=2$, and 
denote by $\mathbb{G}_n$ the prime field of order $p_n$.
\begin{definition}
\begin{enumerate}[(i)]
\item For $n\in \Nat$, define $\mathbb{G}_{n,\mathit{Md}} \subseteq \Nat_{\mathit{Md}}$ by
\[
\mathbb{G}_{n,\mathit{Md}}=\{\underline{i}\mid i< p_n\}.
\]
\item For $n\in \Nat$, define the evaluation $\phi_n:Ter_{\mathit{Md}} \rightarrow \mathbb{G}_{n,\mathit{Md}}$ inductively by
\begin{enumerate}[(i)]
\item $\phi_n(0)=0$, $\phi_n(1)=\underline{1}$,
\item $\phi_n(-t)=\underline{-|\phi_n(t)|\ mod \ p_n}$,
\item \[\phi_n(t^{-1})=
\begin{cases}
0 & \text{ if $|\phi_n(t)| = 0 \ mod \ p_n$}\\
\underline{l} & \text{otherwise, where $0<l<p_n$ and $l|\phi_n(t)|=1 \ mod\ p_n$},
\end{cases}
\]
\item for $\diamond \in \{+,\cdot\}$, $\phi_n(t\diamond t')=\underline{(|\phi_n(t)|\diamond |\phi_n(t')|)\ mod \ p_n}$.
\end{enumerate}
Here we denote by $|\phi_n(t)|$ the corresponding natural number.
\end{enumerate}
\end{definition}
\begin{proposition}\label{soundb}
For $s,t \in Ter_{\mathit{Md}}$ and $n\in \Nat$,
\[
\mathit{Md}\vdash s=t \Rightarrow \phi_n(s)=\phi_n(t).
\]
\end{proposition}
\begin{proof}
Similar to Proposition \ref{sounda}. 
\end{proof}

We now define terms $Z_n$ which equal 0 in any Galois field $\mathbb{G}_m$ with $m\neq n$, and equal 
1 in $\mathbb{G}_n$.
\begin{definition}
For $n\in \Nat$, define $Z_n= 1 - \underline{p_n}\cdot \underline{p_n}^{-1}$.
\end{definition} 
\begin{lemma}\label{Z}
For all $n,m\in \Nat$,
\begin{enumerate}[(i)]
\item $\mathit{Md}\vdash Z_n\cdot Z_n=Z_n$,\label{Z1}
\item $\mathit{Md}\vdash Z_n^{-1}=Z_n$, and  \label{Z2}
\item if $n\neq m$, then $\mathit{Md}\vdash Z_n\cdot Z_m = 0$. \label{Z3}
\end{enumerate}
\end{lemma}
\begin{proof}
Cf. Bergstra and Tucker (2007).
\end{proof}
\begin{lemma}\label{Zmod}
For all $n,m\in \Nat$,
\begin{enumerate}[(i)]
\item $\mathit{Md}\vdash Z_n\cdot \underline{m} = Z_n \cdot \underline{m\ mod\ p_n}$,\label{Zmod1}
\item $\mathit{Md}\vdash Z_n \cdot -\underline{m} = Z_n\cdot \underline{-m\mod\ p_n}$, and \label{Zmod2}
\item $\mathit{Md}\vdash Z_n \cdot \underline{m}^{-1} = Z_n \cdot \underline{l}$
where $l=0$ if $m=0$, or $0<l<p_n$ and $lm=1 \ mod\ p_n$ otherwise. \label{Zmod3}
\end{enumerate}
\end{lemma}
\begin{proof} 
Suppose $m=kp_n+l$ with $0\leq l< p_n$. Then
\[
\begin{array}{rcll}
Z_n\cdot \underline{m}&=& Z_n \cdot \underline{kp_n+l}\\
&=& Z_n\cdot (\underline{kp_n} + \underline{l}) \\
&=&\underline{kp_n} - \underline{p_n}\cdot\underline{p_n}^{-1}\cdot \underline{kp_n} + Z_n\cdot \underline{m\ mod\ p_n}\\
&=&\underline{kp_n} - \underline{p_n}\cdot\underline{p_n}\cdot\underline{p_n}^{-1}\cdot \underline{k} + Z_n\cdot \underline{m\ mod\ p_n}\\
&=&\underline{kp_n} - \underline{kp_n} + Z_n\cdot \underline{m\ mod\ p_n}\\
&=&Z_n\cdot \underline{m\ mod\ p_n}
\end{array}
\]
This proves (1). For (2) observe that
\[
\begin{array}{rcll}
Z_n\cdot \underline{-m\mod\ p_n}&=& Z_n\cdot \underline{p_n -(m\ mod\ p_n)}\\
&=&Z_n \cdot \underline{p_n} - Z_n \cdot \underline{m\ mod\ p_n}\\
&=& -Z_n \cdot \underline{m}&\text{by (1)}\\
&=& Z_n \cdot -\underline{m} 
\end{array}
\]
In order to prove (3) we apply Lemma 2.3 of Bergstra and Tucker (2007), i.e.
\[
Md\vdash u\cdot x\cdot y = u \Rightarrow Md\vdash u\cdot x\cdot x^{-1}=u.
\]
Assume that $ml=1 \ mod \ p_n$. Then
\[
\begin{array}{rcll}
Z_n\cdot \underline{l} &=& Z_n\cdot \underline{1} \cdot\underline{l} \\
&=&Z_n\cdot \underline{lm}\cdot \underline{l}& \text{by the assumption and (1)}\\
&=&(Z_n\cdot \underline{l})\cdot (Z_n\cdot \underline{l}) \cdot(Z_n\cdot \underline{m})&\text{by \ref{Z}.\ref{Z1}.}
\end{array}
\]
Thus 
\[
\begin{array}{rcll}
Z_n\cdot \underline{l}&=&(Z_n\cdot \underline{l}) \cdot (Z_n\cdot \underline{m})\cdot 
(Z_n\cdot \underline{m})^{-1} &\text{ by the lemma}\\
&=& Z_n \cdot \underline{lm} \cdot (Z_n\cdot \underline{m})^{-1}\\
&=& Z_n \cdot \underline{m}^{-1}
\end{array}
\]
\end{proof}
\begin{proposition}\label{zprop}
For all $n\in \Nat$ and $t\in Ter_{\mathit{Md}}$,
\[
\mathit{Md}\vdash Z_n\cdot t = Z_n \cdot \phi_n(t).
\]
\end{proposition}
\begin{proof} 
This follows by structural induction from the previous lemma.
\end{proof}

In addition to the terms $Z_n$, we can define terms $G_n$ such that for all $n$, $G_{n+1}$ equals $0$ in any Galois field with characteristic $p_{n}$ or less;
in any field of characteristic $0$, however,  and in particular, in the zero-totalized field of the rational numbers,
every $G_n$ equals 1. 
\begin{definition}
For $n\in \Nat$, define $G_n \in Ter_{\mathit{Md}}$ inductively as follows:
\begin{enumerate}[(i)]
\item $G_0=1$,
\item $G_{n+1}=G_n\cdot (1-Z_n)$.
\end{enumerate}
Observe that
\[
\mathit{Md}\vdash G_{n+1}= G_n\cdot \underline{p_n}\cdot \underline{p_n}^{-1}.
\]
\end{definition}
\begin{lemma}\label{G}
For all $n, m \in \Nat$ we have
\begin{enumerate}[(i)]
\item $\mathit{Md}\vdash G_n=1-Z_0 - \cdots - Z_{n-1}$,\label{G4}
\item $\mathit{Md}\vdash G_n\cdot Z_n = Z_n$,\label{G5}
\item $\mathit{Md} \vdash G_n =G_n^{-1}$,  \label{G1}
\item $n\leq m \Rightarrow \mathit{Md}\vdash G_m = G_m \cdot G_n$, and \label{G2}
\item if $0<k < p_n$, then $\mathit{Md}\vdash G_n\cdot \underline {k}\cdot\underline{k}^{-1}=G_n$. \label{G3}
\end{enumerate}
\end{lemma}
\begin{proof}
Exercise. For (5) observe that if $0<k<p_n$, then every prime 
factor of $k$ is a factor of $G_n$.
\end{proof}

Clearly, we do not have in general
\[
\mathit{Md}\vdash G_n \cdot t = G_n \cdot \phi(t).
\]
However, we can determine a lower bound in terms of $t$ such that this equation is provable in $\mathit{Md}$ for every $n$ exceeding this bound.
\begin{definition}
We define $\psi: Ter_{\mathit{Md}} \rightarrow \Nat$ inductively as follows.
\begin{enumerate}[(i)]
\item $\psi(0)=0=\psi(1)$,
\item $\psi(-t)=\psi(t)$,
\item $\psi(t^{-1})=\psi(t)$,
\item \[
\psi(t+t')=
\begin{cases}
max\{\psi(t),\psi(t')\} &\text{if $\phi(t)=0$ or $\phi(t')=0$},\\
i& \text{if $|\phi(t)|=\pm \frac{n}{m}$ and $|\phi(t')|=\pm \frac{k}{l}$}
\end{cases}
\]
where $i$ is the least natural number such that $p_i > m,l$,
\item $\psi(t\cdot t')=max\{\psi(t),\psi(t')\}$
\end{enumerate}
\end{definition}
\begin{proposition}\label{gprop}
For each $t\in Ter_{\mathit{Md}}$ and $\psi(t)\leq n\in\Nat$
\[
\mathit{Md}\vdash G_n\cdot t= G_n\cdot \phi(t).
\]
\end{proposition}
\begin{proof} 
It suffices to prove 
\[
\mathit{Md}\vdash G_{\psi(t)}\cdot t= G_{\psi(t)}\cdot \phi(t)
\]
by Lemma \ref{G}.\ref{G2}. We employ structural induction. The base cases are trivial. In the induction step the cases for inversion and 
multiplication follow from Lemma \ref{G}.\ref{G1} - \ref{G2}, and the case for $-t$ from the fact that $\mathit{Md}\vdash \phi(-t) = -\phi(t)$ and
$\psi(-t)=\psi(t)$. For addition, let $t=r+s$ and assume that 
\[
\mathit{Md}\vdash G_{\psi(r)}\cdot r=G_{\psi(r)}\cdot \phi(r)\text{ and }
\mathit{Md}\vdash G_{\psi(s)}\cdot s=G_{\psi(s)}\cdot \phi(s)
\]
Now we distinguish 2 cases.
\begin{enumerate}[(i)]
\item $\phi(r)=0$ or $\phi(s)=0$: Then $\mathit{Md}\vdash \phi(r+s)=\phi(r) + \phi(s)$ and 
$\psi(r+s)=max\{\psi(r), \psi(s)\}$. Thus
\[
\begin{array}{rcll}
\mathit{Md}\vdash G_{\psi(r+s)}\cdot(r+s) & = & G_{\psi(r+s)}\cdot r + G_{\psi(r+s)}\cdot s\\
&=& G_{\psi(r+s)}\cdot G_{\psi(r)}\cdot r + G_{\psi(r+s)}\cdot G_{\psi(s)}\cdot s& \text{ by \ref{G}.\ref{G2}}\\
&=& G_{\psi(r+s)}\cdot G_{\psi(r)}\cdot \phi(r) + G_{\psi(r+s)}\cdot G_{\psi(s)}\cdot \phi(s)\\
&=& G_{\psi(r+s)}\cdot \phi(r) + G_{\psi(r+s)}\cdot \phi(s)\\
&=& G_{\psi(r+s)}\cdot (\phi(r) +\phi(s))\\
&=& G_{\psi(r+s)}\cdot \phi(r + s)
\end{array}
\]
\item $|\phi(r)|=\pm \frac{n}{m}$ and $|\phi(s)|=\pm \frac{k}{l}$: We consider the case that
$\phi(r)=\underline{n}\cdot \underline{m}^{-1}$ and $\phi(s)=\underline{k}\cdot \underline{l}^{-1}$. 
First observe that since $p_{\psi(r+s)}>m,l$, $p_{\psi(r+s)}$ exceeds every common prime factor of 
$ml$ and $nl+km$.
 Thus 
\[
\begin{array}{rcll}
\mathit{Md}\vdash G_{\psi(r+s)}\cdot(r+s) & = & G_{\psi(r+s)}\cdot r + G_{\psi(r+s)}\cdot s\\
&=&G_{\psi(r+s)}\cdot \underline{n}\cdot \underline{m}^{-1} + G_{\psi(r+s)}\cdot \underline{k}\cdot \underline{l}^{-1}\\
&=&G_{\psi(r+s)}\cdot \underline{l}\cdot \underline{l}^{-1}\cdot \underline{n}\cdot \underline{m}^{-1} +
G_{\psi(r+s)}\cdot \underline{m}\cdot \underline{m}^{-1}\cdot \underline{k}\cdot \underline{l}^{-1}& \text{ by \ref{G}.\ref{G3}}\\
&=&G_{\psi(r+s)}\cdot (\underline{l}\cdot \underline{n} +
\underline{m}\cdot \underline{k})\cdot (\underline{m}\cdot \underline{l})^{-1}\\
&=&G_{\psi(r+s)}\cdot \underline{l n +
mk}\cdot \underline{ml}^{-1}\\
&=&G_{\psi(r+s)}\cdot \phi(r + s)
\end{array}
\]
\end{enumerate}
by repeated removal of shared prime factors using again \ref{G}.\ref{G3}.
The remaining  
3 cases follow by a similar argument taking in addition the sign into account. 
\end{proof}

We are now able to determine for every closed meadow term the normal form mentioned in the beginning of this section.
\begin{proposition}\label{normal}
For each $t\in Ter_{\mathit{Md}}$,
\[
\mathit{Md}\vdash t=\Sigma_{i=0}^{\psi(t)-1}Z_i\cdot \phi_{i}(t) + G_{\psi(t)}\cdot \phi(t).
\]
\end{proposition}
\begin{proof} 
First observe that ($*$)
\[
\begin{array}{rcll}
G_n\cdot t &=& (Z_n + (1- Z_n))\cdot G_n \cdot t \\
&=& Z_n \cdot G_n \cdot t + (1-Z_n)\cdot G_n \cdot t\\
&=& Z_n \cdot G_n \cdot t + G_{n+1} \cdot t\\
&=& G_n \cdot Z_n \cdot \phi_n(t) + G_{n+1} \cdot t& \text{ by Proposition \ref{zprop}}\\
&=& Z_n \cdot \phi_n(t) + G_{n+1} \cdot t & \text{ by Lemma \ref{G}.\ref{G5}.}
\end{array}
\]
We therefore can expand $t$ as follows:
\[
\begin{array}{lcll}
t &=& G_0 \cdot t\\
&=&Z_0 \cdot \phi_0(t) + G_{1} \cdot t\\
&\vdots&\\
&=&Z_0 \cdot \phi_0(t) + \cdots +  Z_{\psi(t) -1} \cdot \phi_{\psi(t) - 1 }(t) + G_{\psi(t)} \cdot t& \text{ by repeated use of ($*$)}\\
&=&Z_0 \cdot \phi_0(t) + \cdots +  Z_{\psi(t) -1} \cdot \phi_{\psi(t) - 1 }(t)+ G_{\psi(t)}\cdot \phi(t)& \text{ by the previous proposition.}
\end{array}
\]
\end{proof}
\begin{proposition}
For each $t\in Ter_{\mathit{Md}}$ and $\psi(t)\leq n\in \Nat$,
\[
\mathit{Md}\vdash t=\Sigma_{i=0}^{n-1}Z_i\cdot \phi_{i}(t) + G_{n}\cdot \phi(t).
\]
\end{proposition}
\begin{proof} 
By expanding $t$ as far as necessary and Proposition \ref{gprop} .
\end{proof}
\begin{theorem}
For all $s,t\in Ter_{\mathit{Md}}$,
\[
\mathit{Md}\vdash s=t  \Leftrightarrow \text{ for all } i\leq max {\{\psi(s),\psi(t)\}-1}\ 
\ \phi_{i}(s)=\phi_{i}(t)\ \&\ \phi(s)=\phi(t)
\]
\end{theorem}
\begin{proof} 
Left to right follows from Propositions \ref{sounda} and \ref{soundb}. 
For the reverse direction apply the previous proposition. 
\end{proof}
\begin{corollary}
The closed word problem for meadows is decidable. 
\end{corollary}
In Bergstra and Tucker (2007) it is proved that the closed equational theories of 
zero-totalized fields and of meadows coincide. Thus decidability of the closed word
problem for meadows carries over to zero-totalized fields.
\begin{corollary}
The closed word problem for zero-totalized fields is decidable.
\end{corollary}

\section{Conclusion}
We have represented the initial meadows as follows:
\begin{enumerate}[(i)]
\item the initial meadow of characteristic $0$ is the minimal submeadow of the direct product of all
finite prime fields|it is a proper submeadow and not a product of fields|and
\item the initial meadow of characteristic $k>0$ is $\prod_{p \text{ with }p|k}\mathbb{G}_p$.
\end{enumerate}
This gives a clear picture of
the finite and infinite initial objects in the categories of meadows.

The finite initial meadows are decidable and so is the infinite one. The open word problem, however, 
remains open. In particular, it is not known whether there exists a finite Knuth-Bendix completion of the specification of  meadows.

\section*{Acknowledgement} 
This work was partly
 supported by The Netherlands Organisation for Scientific Research (NWO) under grant
 638.003.611.

\end{document}